\documentclass[12pt,a4paper]{article}
\usepackage{graphicx} 
\usepackage{mathrsfs,amssymb,amsfonts,amsthm,amsmath,amsbsy,fullpage,mathtools}
\usepackage{mathtools}
\usepackage{mathabx}
\usepackage{graphicx}
\usepackage{wrapfig}
\usepackage{caption}
\usepackage{subfig}
\usepackage{dsfont}
\usepackage{float}
\usepackage{upgreek}
\usepackage{listings}
\usepackage{lipsum}
\usepackage{bbm, dsfont}

\newtheorem{theorem}{Theorem}[section]

\newtheorem{lemma}[theorem]{Lemma}
\newtheorem{lem}[theorem]{Lemma}

\newtheorem{thm}[theorem]{Theorem}

\newtheorem{rem}[theorem]{Remark}

\newtheorem{dfn}[theorem]{Definition}

\usepackage{comment}

\usepackage{inlinelabel}
\allowdisplaybreaks

\title{Largest-loop-first loop-erased random walk on $\mathbb{Z}^{4}$}
\author{Daisuke  Shiraishi\footnote{Graduate School of Informatics, Kyoto University, shiraishi@acs.i.kyoto-u.ac.jp.} \ \  and \ \  Satomi  Watanabe\footnote{Research Institute for Mathematical Sciences, Kyoto University, swatanab@kurims.kyoto-u.ac.jp.}}
\date{}

\begin{document}

\maketitle

\begin{abstract}
    Let $S = \big(S(n)\big)_{n \ge 0}$ be a simple random walk on $\mathbb{Z}^{d}$ started at the origin. We study a loop-erasing procedure of $S[0,n]$ that differs from Lawler’s chronological loop-erasure. Specifically, we remove loops from $S[0,n]$ in decreasing order of their lengths. The resulting random simple path is called the largest-loop-first (LLF) LERW. For $d=4$, we prove that the expected length of LLF LERW is of the order $n (\log n)^{-1/2 + o(1)}$. In particular, this suggests that chronological LERW and LLF LERW belong to different universality classes. Furthermore, we also prove the convergence of LLF LERW to Brownian motion in four dimensions.
\end{abstract}

\section{Introduction}
Loop-erased random walk (LERW) is a simple random path obtained from a random walk trajectory by removing loops in chronological order. In this paper, we refer to this procedure as chronological loop-erasure and accordingly call the resulting process chronological LERW. For the precise definition of chronological loop-erasure, see Definition \ref{lawler} below. Since its introduction by Lawler in \cite{Lawloop}, it has been shown that the chronological LERW is closely connected to several fundamental probabilistic models in statistical mechanics. For example:

\begin{itemize}

\item The critical dimension of the chronological LERW is four, as is expected for the self-avoiding walk (see  \cite{MadSla} for example).  It is shown in \cite{Lawb} that for $d \ge 4$, the scaling limit of the chronological LERW is the Brownian motion.

\item It is closely related to the uniform spanning tree (\cite{Pema}, \cite{Wils}).

\item In two dimensions, it exhibits conformal invariance and the scaling limit is described by Schramm–Loewner evolution (\cite{Keny}, \cite{LSW}, \cite{LawVik}, \cite{Schr}).

\item Brownian motion can be recovered by adding an independent Brownian loop soup to the scaling limit of the chronological LERW (\cite{LawWer}, \cite{SaShi}). 

\end{itemize}
As demonstrated by the self-avoiding walk, the study of random simple paths is, in general, fraught with difficulties. Among these, however, chronological LERW is one for which relatively detailed properties can be established, and it has been studied extensively. It is worth noting that one can obtain fairly sharp results for the chronological LERW even in three dimensions. For example, the scaling limit of the three-dimensional chronological LERW has been proved to exist not only as a set  (that is, under the topology induced by the Hausdorff distance) in \cite{Koz}, but also as a stochastic process  (that is, under the topology of the uniform convergence norm) in \cite{LiHS} and \cite{LiS}. Convergence in the sense of stochastic processes plays an essential role in \cite{ACHS} in proving the existence of the scaling limit of the three-dimensional uniform spanning tree under a Gromov–Hausdorff–Prokhorov-type topology.

In this paper, we consider a loop-erasing procedure for random walk paths that differs from the chronological one. Specifically, among the loops formed by a random walk trajectory, we remove them successively in decreasing order of their lengths (see Definition \ref{LLFLE} for the precise definition). We refer to this procedure as largest-loop-first (LLF) loop-erasure, and call the resulting random simple path LLF LERW.

Although this type of loop-erasing procedure has been studied on certain special fractals such as  Sierpi\'{n}ski gasket graphs (\cite{Hatt}, \cite{HKN}, \cite{HatMiz}), to our knowledge, LLF LERW on $\mathbb{Z}^{d}$ remains far from being well understood. As stated in \cite{HKN}, the distributions of the chronological LERW and the LLF LERW coincide on Sierpi\'{n}ski gasket graphs (see also \cite{Cao} for related work). In contrast, it is not obvious whether chronological LERW and LLF LERW defined on $\mathbb{Z}^{d}$ belong to the same universality class.
In \cite{Law4d}, it is proved that the expected length of the chronological loop-erasure of $S[0,n]$ is of order $n(\log n)^{-\frac{1}{3}}$ for $d=4$ (for further results on ``mean-field'' behavior in four-dimensional chronological LERW and uniform spanning trees, see also \cite{BenKoz}, \cite{HalHut}, \cite{HutSou}, \cite{Law4d-2}, \cite{Schwe}).
In this paper, we prove that the expected length of LLF LERW on $\mathbb{Z}^{4}$ exhibits a logarithmic correction different from that of chronological LERW (see Theorem \ref{main} below). In particular, this suggests that chronological LERW and LLF LERW belong to different universality classes.

Let $S = \big( S(n) \big)_{n \ge 0}$ be a simple random walk on $\mathbb{Z}^{d}$ started at the origin. Write $L_{n}$ for the length of the LLF loop-erasure of $S[0,n]$ (see \eqref{Ln} for the precise definition). We now state our main result.

\begin{thm}\label{main}

Let $d=4$. As $n \to \infty$, the expectation of $L_{n}$ satisfies

\begin{equation*}
E (L_{n}) = n (\log n)^{-\frac{1}{2} + o(1)} .
\end{equation*}

\end{thm}

\begin{rem}
Let ${\cal C}_{n}$ be the number of cut times for $S$ up to time $n$ (see Definition \ref{3-def1} for this). 
Regarding $S[0,n]$  as a random graph whose vertex set consists of the points visited by $S$ up to time $n$, and whose edge set consists of the edges traversed up to time $n$, one can consider the graph distance and the effective resistance on this graph. We write $D_{n}$ and $R_{n}$ for the graph distance and the effective resistance between the origin and $S(n)$ on $S[0, n]$, respectively. Note that 
$$ {\cal C}_{n} \le R_{n} \le D_{n} \le L_{n}. $$

For $d=4$, it is proved in \cite{slowly} that $E ({\cal C}_{n})$ is asymptotic to $c_{1} n (\log n)^{-\frac{1}{2}}$, while it is shown in \cite{DS-SW-4d} that $E (R_{n})$ and $E (D_{n})$ are asymptotic to $c_{2} n (\log n)^{-\frac{1}{2}}$ and $c_{3} n (\log n)^{-\frac{1}{2}}$, respectively.
Here, $0 < c_{1} \le c_{2} \le c_{3} < \infty$ are some positive constants.
    
\end{rem}

Let
$$\gamma = \text{LE}^{\ast} \big( S [0, n] \big) \ \ \text{ and } \ \  \phi (n) = \frac{E (L_{n} )}{n}$$
denote, respectively, the LLF loop-erasure of $S[0,n]$ and the logarithmic correction appearing in the expected length of $\gamma$.
By linear interpolation, we regard $\gamma$ as a continuous $\mathbb{R}^{4}$-valued function defined for all $0 \le t \le L_{n}$. Define the process $X_{n}$ obtained by rescaling $\gamma$ as follows:
\begin{equation}\label{xnt}
    X_{n} (t) = \frac{ 2 \gamma \Big( \min \big\{ n \phi (n) t, L_{n} \big\} \Big) }{\sqrt{n}} \ \ \ \ \text{ for \  $0 \le t \le 1$.}
\end{equation}

The following theorem establishes the convergence of LLF LERW to Brownian motion.

\begin{thm}\label{main2} 
    Let $d=4$. As $n \to \infty$, the law of 
    $\big( X_{n}(t) \big)_{t \in [0,1]} $ converges weakly to the law of a standard Brownian motion $\big( B(t) \big)_{t \in [0,1]}$ in $\mathbb{R}^{4}$ started at the origin (with respect to the topology of uniform convergence on $[0,1]$).
\end{thm}

We make some remarks here regarding the proof of Theorem \ref{main}. The proof of Theorem \ref{main} follows the same strategy as the proof of Theorem 1.2.3 in \cite{exact} and that of Theorem 1.1 in \cite{DS-SW-4d}. Namely, to prove Theorem \ref{main}, it is sufficient to verify the following two conditions:

\begin{itemize}

\item[(A1)] $\phi(n)$ is slowly varying;

\item[(A2)] for all $\epsilon \in (0,1)$ there exists $N = N_{\epsilon} \ge 1$ such that for every $n \ge N$ we have

\begin{equation}\label{ass2}
\phi (2n) \le \phi (n) \Big\{ 1 - \frac{\log 2}{2 \log n} (1 - \epsilon ) \Big\}. 
\end{equation}

\end{itemize}
See Lemma \ref{exactlem} for this.

In \cite{exact} and \cite{DS-SW-4d}, checking the inequality \eqref{ass2} for $E(D_{n})/n$ and $E(R_{n})/n$ relies essentially on the fact that the graph distance and the effective resistance are metrics. 
The main difficulty in proving Theorem \ref{main} is that the length of the LLF loop-erasure does not satisfy the properties of a metric, and therefore special care is required in establishing \eqref{ass2}. Indeed, the logarithmic correction of the expected length of the chronological loop-erasure of $S[0,n]$ is slowly varying, but it does not satisfy the inequality \eqref{ass2}. In the present paper, this difficulty is overcome---particularly in the proof of Lemma \ref{3-lem3}---by exploiting specific features of the LLF loop-erasure, thereby allowing us to establish the inequality \eqref{ass2}.

For $d=2,3$, the study of LLF LERW remains largely unexplored. Since LLF LERW does not satisfy a domain Markov property in a simple form comparable to that of chronological LERW, this could pose an obstacle to further analysis. On the other hand, LLF loop-erasure also has an advantage that chronological loop-erasure does not share. Namely, LLF loop-erasure can be defined directly for trajectories of Brownian motion, although the LLF loop-erasing procedure for Brownian motion requires infinitely many steps to terminate.

In the cases $d=2,3$, one may consider, for example, the following problems.

\begin{itemize}

\item Prove the existence of a growth exponent $\alpha_d$ such that

$$E(L_{n}) = n^{\alpha_{d} + o(1)},$$

and, if possible, determine the value of $\alpha_d$. 
\begin{itemize}
     \item For the chronological LERW, the value of the growth exponent is $5/8$ in two dimensions (\cite{Keny}).

     \item The existence of the growth exponent has been established for the chronological LERW in three dimensions (\cite{growth}), while the value of the exponent is not known. It has been proved that the exponent is strictly larger than $1/2$ and at most $5/6$ (\cite{Law99}).
\end{itemize}

\item Compare $\alpha_d$ with the growth exponent of chronological LERW. Furthermore, compare the orders of $E(L_n)$ and $E(D_n)$. In fact, this problem motivated our consideration of LLF LERW.

 \end{itemize}
We do not address these problems in this paper and leave them for future work.

The organization of this paper is as follows. In Section 2, we introduce the notation and definitions used throughout the paper. We prove Theorem \ref{main} in Section 3 and Theorem \ref{main2} in Section 4.

\subsection*{Acknowledgments}
DS was supported by JSPS Grant-in-Aid for Scientific Research (C) 22K03336, JSPS Grant-in-Aid for Scientific Research (B) 22H01128
and 21H00989. SW was supported by JSPS Grant-in-Aid for Research Activity Start-up 25K23335.

\section{Notation and preliminaries}
In this section, we introduce the notation and definitions used throughout this paper. In particular, in Definition \ref{LLFLE}, we define LLF loop-erasure. This is a method of erasing loops that is distinct from the chronological loop-erasure introduced by Lawler in \cite{Lawloop}.

We use $\mathbb{Z}^d$ to denote the $d$-dimensional integer lattice and $\mathbb{R}^d$ for the $d$-dimensional Euclidean space. For any $x \in \mathbb{R}^d$, let $|x|$ represent the Euclidean distance from $x$ to the origin. In this paper, we restrict attention to the case $d = 4$.

Let $S$ be a simple random walk on $\mathbb{Z}^{d}$, and denote by $S(j)$ its position at time $j \ge 0$. When the random walk starts from $x$, i.e., $S(0) = x$, we write $P^x$ and $E^x$ for its law and expectation, respectively. For convenience, we do not explicitly use $\lfloor r \rfloor$ even when $r \ge 0$ is not an integer. In such cases, expressions like $S(r)$ are understood to mean $S(\lfloor r \rfloor)$.

Throughout the paper, $c$, $C$, $c_1$, $c_{2}$, and so on denote positive constants whose values may vary from one occurrence to another. To indicate the dependence of a constant on a parameter $\epsilon$, we write $c_{\epsilon}$ or $C_{\epsilon}$. Let $\{a_n \}_{n \ge 1}$ be a sequence that satisfies $a_n > 0$ for all $n \ge 1$. For another sequence $\{b_n \}_{n \ge 1}$, we write $b_n = o(a_n)$ if $\lim_{n \to \infty} b_n / a_n = 0$, and $b_n = O(a_n)$ if there exists a constant $C > 0$ such that $|b_n| \le C a_n$ for all $n \ge 1$. We set $a_{n} \sim b_{n}$ if $\lim_{n \to \infty} b_n / a_n = 1$, and write $a_{n} \asymp b_{n}$ if there exists a constant $C>0$
 such that  $C^{-1} a_{n} \le b_{n} \le C a_{n}$ for all $n \ge 1$. 
 When it is necessary to indicate the dependence on a parameter $\epsilon$, we use the notation  $O_{\epsilon}(a_n)$ and $a_{n} \asymp_{\epsilon} b_{n}$, where the implicit constants depend on $\epsilon$.

The notation and definitions concerning paths in $\mathbb{Z}^{d}$ are given below.

\begin{dfn}\label{LE}
 Let $m \ge 0$ be an integer. A sequence of points 
 $$\lambda = [\lambda(0), \lambda(1), \cdots, \lambda(m)] \subset \mathbb{Z}^{d}$$ is called a path of length $m$ if $|\lambda(k) - \lambda(k+1)| = 1$ for all $0 \le k \le m-1$. We write $\text{len}\, (\lambda) = m$ for the length of $\lambda$. For $0 \le a \le b \le m$, we set $\lambda[a,b] = [\lambda(a), \lambda(a+1), \cdots, \lambda(b)]$ to denote the portion of the path $\lambda$ between times $a$ and $b$. Furthermore, $\lambda$ is said to be a simple path if $\lambda(j) \neq \lambda(k)$ for all $0 \le j < k \le m$.

     Let $\lambda_{1} = [\lambda_{1}(0), \cdots, \lambda_{1}(m_{1})]$ and $\lambda_{2} = [\lambda_{2}(0), \cdots, \lambda_{2}(m_{2})]$ be two paths in $\mathbb{Z}^{d}$ of lengths $m_{1}$ and $m_{2}$, respectively, and suppose that $\lambda_{1}(m_{1}) = \lambda_{2}(0)$. We define the concatenation of $\lambda_{1}$ and $\lambda_{2}$, denoted by $\lambda_{1} \oplus \lambda_{2}$, by
$$\lambda_{1} \oplus \lambda_{2} = [ \lambda_{1} (0), \cdots , \lambda_{1} (m_{1}) , \lambda_{2} (1), \cdots , \lambda_{2} (m_{2})]. $$
    Then $\lambda_{1} \oplus \lambda_{2}$ is a path of length $m_{1}+m_{2}$.

\end{dfn}

Lawler's definition (given in Section 7.2 of \cite{Lawb}) of chronological loop-erasure is as follows.

\begin{dfn}\label{lawler}
Let $\lambda = [\lambda(0), \lambda(1), \cdots, \lambda(m)]$ be a path of length $m$ in $\mathbb{Z}^{d}$. We define the chronological loop-erasure of $\lambda$, denoted by $\text{LE} (\lambda)$, in the following manner. First, set
$$ t_0 \coloneqq \max\{t \ge 0 \, : \, \lambda(t) = \lambda(0)\}.$$
For each $i \geq 1$, define
$$ t_i \coloneqq \max\{t \ge 0 \, : \, \lambda(t) = \lambda(t_{i-1}+1)\}. $$
Let 
$$n = \min\{i \, : \, t_i = m\}.$$
    Then, the chronological loop-erasure of $\lambda$ is given by
    $$\text{LE} (\lambda) \coloneqq [\lambda(t_0), \lambda(t_1), \ldots, \lambda(t_n)].$$
\end{dfn}

Next, as an alternative loop-erasing procedure, we define LLF loop-erasure as follows.
    
\begin{dfn}\label{LLFLE}
    Let $\lambda = [\lambda(0), \lambda(1), \cdots, \lambda(m)]$ be a path of length $m$ in $\mathbb{Z}^{d}$. 
    \begin{itemize}
    \item[(i)] Define
    $$ l_{1} = \max \big\{ k -j \ : \ 0 \le j \le k \le m, \ \lambda (j) = \lambda (k) \big\}. $$
This is the maximum length of loops in $\lambda$. Among the pairs $(j,k)$ that attain the maximum on the right-hand side of the definition of $l_{1}$ above, let $(j_{1},k_{1})$ be the one with the smallest $j$:
\begin{itemize}
\item let 
$$\Lambda (1) = \Big\{ (j, k) \, : \, 0 \le j \le k \le m, \ \ \lambda (j) = \lambda (k), \ \ k-j = l_{1} \Big\},$$
\item set 
$$ j_{1} = \min \big\{ 0 \le  j \le m \, : \, \exists k \in [j, m] \, \text{ such that } \, (j,k) \in \Lambda (1) \big\},$$
\item choose $k_{1}$ so that 
$(j_{1}, k_{1}) \in \Lambda (1). $
\end{itemize}
That is, $(j_{1},k_{1})$ is the leftmost pair $(j,k)$ such that $\lambda(j)=\lambda(k)$ and $k-j=l_{1}$.

\item[(ii)]
 We define $\lambda_{1}$ to be the path obtained from $\lambda$ by removing this maximal loop. That is, $\lambda_{1}$ is a path of length $m - k_{1} + j_{1}$ defined by 
$$ \lambda_{1} = [\lambda (0), \cdots, \lambda (j_{1}), \lambda (k_{1} +1), \cdots \lambda (m) ] = \lambda [0, j_{1}] \oplus \lambda [k_{1}, m]. $$
Note that $\lambda = \lambda_{1}$ if and only if $\lambda$ is a simple path. For notational convenience, we abbreviate $\lambda$ as $\lambda_0$.

\item[(iii)] Apply the procedures (i) and (ii) above to $\lambda_{1}$ in the same way. Denote the resulting path by $\lambda_{2}$. Repeating the same procedure, $\lambda_{3}, \lambda_{4}, \cdots$ are obtained.

\item[(iv)] Define $q$ by
$$ q = \min \Big\{ l \ge 0 \ : \ \text{ $\lambda_{l}$ is a simple path} \Big\}.$$
 The largest-loop-first (LLF) loop-erasure of $\lambda$ is given by 
 $$ \text{LE}^{\ast} (\lambda) = \lambda_{q}. $$

\end{itemize}
    
\end{dfn}

\begin{rem}
The notation is the same as in Definition \ref{LE}. As in the case of the chronological loop-erasure (for basic properties of chronological loop-erasure, see Section 7.2 of \cite{Lawb}), $\mathrm{LE}^{\ast}(\lambda)$ is a simple path satisfying $\mathrm{LE}^{\ast}(\lambda) \subset \lambda$, whose initial and terminal points coincide with $\lambda(0)$ and $\lambda(m)$, respectively. In particular, we have
\begin{equation}\label{trivial}
\text{len} \, \Big( \text{LE}^{\ast} (\lambda ) \Big) \le m. 
\end{equation}

Furthermore, if $0 \le k \le m$ is a cut time of $\lambda$, that is, if $\lambda[0,k] \cap \lambda[k+1,m] = \emptyset$, then one can check that
\begin{equation}\label{CUT-LE}
    \text{LE}^{\ast} (\lambda) = \text{LE}^{\ast} \big( \lambda [0, k] \big) \oplus \text{LE}^{\ast} \big( \lambda [k, m] \big).
    \end{equation}
    
\end{rem}

\section{Proof of Theorem \ref{main}}

In this section, we present a proof of the main theorem, Theorem \ref{main}. To this end, we first establish several lemmas. As noted in Section 1, a property of LLF loop-erasure that is not shared by chronological loop-erasure is employed in the proof of Lemma \ref{3-lem3} below.

For a simple random walk $S$ on $\mathbb{Z}^d$ starting from the origin, let 
\begin{equation}\label{Ln}
 L_{n} = \text{len} \, \Big( \text{LE}^{\ast} \big( S[0,n] \big) \Big) 
 \end{equation}
be the length of the LLF loop-erasure of $S[0,n]$. Our goal in this section is to prove the following theorem.

\begin{thm}\label{LE-result}
Let $d=4$. We have 
\begin{equation*}
    E (L_{n} ) = n (\log n)^{-\frac{1}{2} + o (1)}.
\end{equation*}

\end{thm}

Write
\begin{equation*}
\phi (n) = \frac{E (L_{n})}{n}    
\end{equation*}
for the logarithmic correction of $E (L_{n})$. By combining an estimate of the expected number of cut times with \eqref{CUT-LE}, we have
\begin{equation}\label{3-eq-2}
\phi (n) \ge c (\log n)^{-\frac{1}{2}},
    \end{equation}
for some $c > 0$. See Remark \ref{remcut} below for this. Therefore, to prove Theorem \ref{LE-result}, we can follow the same method as in the proof of Theorem 1.2.3 in \cite{exact}. Namely, we aim to use the following lemma.

\begin{lem}[Lemma 4.1.1 in \cite{exact}]\label{exactlem}
    Suppose that $\phi (n)$ satisfies the following two conditions (A1) and (A2).
\begin{itemize}
   \item[(A1)] $\phi (n)$ is slowly varying.
    \item[(A2)] For all $\epsilon \in (0,1)$ there exists $N = N_{\epsilon}$ such that for every $n \ge N$ we have 
    \begin{equation}\label{A2}
     \phi (2n) \le \phi (n) \Big\{ 1 - \frac{\log 2}{2 \log n} (1 - \epsilon ) \Big\}. 
     \end{equation}
\end{itemize}
    Then for all $\epsilon \in (0,1)$ we have 
    $$ \limsup_{n \to \infty} \, (\log n)^{\frac{1}{2} - \epsilon} \phi (n) = 0. $$
\end{lem}

\begin{rem}
Although only the condition (A2) is stated in the assumption of Lemma 4.1.1 in \cite{exact}, its proof makes use of the fact that a function is slowly varying.
    \end{rem}

Since we have the lower bound in \eqref{3-eq-2}, it remains to check the conditions (A1) and (A2) in Lemma \ref{exactlem}. To this end, we need to prepare several auxiliary lemmas. We begin by introducing the notation.

\begin{dfn}\label{3-def1}

\begin{itemize}
   \item[(D1)] Let $b_{n} = n (\log n)^{-2}$ and $M = (\log n)^{2}.$ We divide the time interval $[0,n]$ into $M$ subintervals of equal length ($= b_{n}$) as follows. Let 
$$I^{1} (j) = \big[ n - j b_{n}, n - (j-1) b_{n} \big] =: [t^{1}_{j}, t^{1}_{j-1}] \ \ \ \text{ for \ \  $j =1, 2, \cdots , M$.} $$
 Note that $ 0 = t^{1}_{M} < t^{1}_{M-1} < \cdots < t^{1}_{1} < t^{1}_{0} = n. $

Similarly, we divide the time interval $[n,2n]$ in the same way as follows. Let 
$$I^{2} (j) = \big[ n + (j-1) b_{n}, n + j b_{n} \big] =: [t^{2}_{j-1}, t^{2}_{j}] \ \ \ \text{ for \ \  $j =1, 2, \cdots , M$.}$$
 Note that $ n= t^{2}_{0} < t^{2}_{1} < \cdots < t^{2}_{M-1} < t^{2}_{M} = 2n.  $

\item[(D2)] We set 

$$ L^{1} (j) = \text{len} \, \Big( \text{LE}^{\ast} \big( S [t^{1}_{j}, t^{1}_{j-1}  ] \big) \Big) \ \ \ \text{and } \ \ \ L^{2} (j) = \text{len} \, \Big( \text{LE}^{\ast} \big( S [t^{2}_{j-1}, t^{2}_{j}  ] \big) \Big). $$

As in (20) of \cite{DS-SW-4d}, we let 
$$ L^{\ast} = \max \big\{ j + k \ : \ I^{1} (j) \leftrightarrow I^{2} (k) \big\} $$
denote the size of the longest intersection. Here, for two intervals $I, J \subset [0, \infty)$, we write 
$$ I \leftrightarrow J \ \ \text{ if } \ \ \{ S(i) \ : \ i \in I \} \cap \{ S(j) \ : \ j \in J \} \neq \emptyset.$$
Clearly, $ 2 \le L^{\ast} \le 2M $. 


\item[(D3)] Set 
$$ L_{n}' = \text{len} \, \Big( \text{LE}^{\ast} \big( S[n,2n] \big) \Big). $$
Note that $L_{n}$ and $L_{n}'$ have the same distribution.

\item[(D4)] A time $k \in [0,n]$ is called a cut time (up to $n$) if $S[0,k] \cap S[k+1,n] = \emptyset$.

\end{itemize}
\end{dfn}

\begin{rem}\label{remcut}
For $d=4$, if we let ${\cal C}_{n}$  denote the number of cut times up to $n$ in the time interval $[0,n]$, it is shown in \cite{slowly} that 
\begin{equation*}
    E \big( {\cal C}_{n} \big) = c  n (\log n)^{-\frac{1}{2}} \big\{ 1 + o (1) \big\},
\end{equation*}
for some $c >0$.
 Combining this and \eqref{CUT-LE}, we have
\begin{equation}\label{cutbound-1}
E (L_{n} ) \ge c n (\log n)^{-\frac{1}{2}}.
\end{equation}
Here, we recall that $L_{n}$ is defined in \eqref{Ln}. Dividing both sides of \eqref{cutbound-1} by $n$, we obtain \eqref{3-eq-2}.
\end{rem}

The following lemma provides a bound on the difference between $L_n$ and its expectation.

\begin{lemma}
Let $d=4$. For any $\epsilon \in (0, 1)$,  there exist $C = C_{\epsilon} > 0$ and $N = N_{\epsilon} \ge 1$ depending on $\epsilon$ such that for all $n \ge N$, we have
\begin{equation}\label{3-lem1-1}
    P \Big( \big| L_{n} - n \phi (n) \big| \ge \epsilon n \phi (n) \Big) \le C (\log n )^{-\frac{1}{2}} \log \log n.
\end{equation}
\end{lemma}

\begin{proof}

  Let 
$$a_{n} = n (\log n)^{-6}. $$
For $j = 1,2, \dots, M$, let $J_{j}$ be the indicator function of the event that there exists no cut time up to $n$ in at least one of the time intervals $\big[ t^{1}_{j}, t^{1}_{j} + a_{n} \big]$ or $\big[ t^{1}_{j-1} - a_{n}, t^{1}_{j-1} \big]$. Then, by \eqref{trivial} and \eqref{CUT-LE}, there exists a universal constant $C>0$ such that
\begin{equation}\label{lem-1-2}
    \Big| L_{n} - \sum_{j=1}^{M} L^{1} (j) \Big| \le C \Big( n (\log n)^{-4} + n (\log n)^{-2} \sum_{j=1}^{M} J_{j} \Big),
\end{equation}
where $L^{1} (j)$ is defined in (D2) of Definition \ref{3-def1}. Let us make a brief remark concerning this inequality. When $J_{j} = 0$, let $T^{-}_{j} \in \big[ t^{1}_{j}, t^{1}_{j} + a_{n} \big]$ and $T^{+}_{j} \in \big[ t^{1}_{j-1} - a_{n}, t^{1}_{j-1} \big]$ denote cut times whose existence is guaranteed. For an interval $I^{1}(j)$ with $J_{j} = 0$, by decomposing $\text{LE}^{\ast} \big( S[0,n] \big)$ and $\text{LE}^{\ast} \big( S[t^{1}_{j}, t^{1}_{j-1}] \big)$ using \eqref{CUT-LE}, we see that the term $\text{len} \, \Big( \text{LE}^{\ast} \big( S[T^{-}_{j}, T^{+}_{j} ] \big) \Big)$ appears twice in the difference on the left-hand side of \eqref{lem-1-2} and cancels out. On the other hand,  for $I^{1}(j)$ with $J_{j}=1$, by using \eqref{trivial}, we see that the contribution coming from each such interval is bounded above by $ Cb_{n} = Cn (\log n)^{-2}$. As a result, the inequality \eqref{lem-1-2} follows.

Note that Lemma 7.7.4 of \cite{Lawb} guarantees that 
$$ E (J_{j} ) \le C (\log n)^{-1} \log \log n, $$
for each $j=1,2, \cdots, M.$ Combining this with  \eqref{lem-1-2}, it follows that 
\begin{equation}\label{expec}
    \Big| E \Big( \sum_{j=1}^{M} L^{1} (j) \Big) - n \phi (n) \Big| \le C n (\log n)^{-1} \log \log n.
\end{equation}
This implies that 
\begin{equation}\label{diffphi}
    \big| \phi (n) - \phi \big( n (\log n)^{-2} \big) \big| \le C (\log n)^{-1} \log \log n.
\end{equation}
In particular, as $\phi (n) \ge c (\log n)^{-\frac{1}{2}}$ by \eqref{cutbound-1}, we have 
\begin{equation}\label{slowvar}
    \phi \big( n (\log n)^{-2} \big) = \phi (n) \Big\{ 1 + O \Big( (\log n)^{-\frac{1}{2}} \log \log n \Big) \Big\}.
\end{equation}

On the other hand, using \eqref{trivial} and the independence of the $L^{1}(j)$'s, it follows from \eqref{expec}  that 
$$ \text{Var} \Big( \sum_{j=1}^{M} L^{1} (j) \Big) \le  \sum_{j=1}^{M} E \Big(  L^{1} (j)^{2} \Big) \le n (\log n)^{-2} \sum_{j=1}^{M} E \Big(  L^{1} (j) \Big) \le C n^{2} (\log n)^{-2} \phi (n). $$

Let $(\log n)^{-\frac{1}{3}} \le \epsilon < 1 $. Since $\phi (n) \ge c (\log n)^{-\frac{1}{2}},$ it follows from \eqref{lem-1-2} and \eqref{diffphi} that for large $n$

\begin{align}
&P \Big( \big| L_{n} - n \phi (n) \big| \ge \epsilon n \phi (n) \Big) \notag \\
&\le P \Big( C n (\log n)^{-2} \sum_{j=1}^{M} J_{j} \ge \frac{\epsilon}{4} n \phi (n) \Big) + P \Big( \Big| \sum_{j=1}^{M} L^{1} (j)  - n \phi (n) \Big| \ge \frac{\epsilon}{4} n \phi (n) \Big) \notag \\
&\le C \epsilon^{-1} (\log n)^{-\frac{1}{2}} \log \log n + P \bigg( \bigg| \sum_{j=1}^{M} L^{1} (j)  - \sum_{j=1}^{M} E \Big( L^{1} (j) \Big) \bigg| \ge \frac{\epsilon}{4} n \phi (n) \bigg) \notag \\
&\le C \epsilon^{-1} (\log n)^{-\frac{1}{2}} \log \log n + C \epsilon^{-2} (\log n)^{-\frac{3}{2}},  \label{diffest}
\end{align}
where $C > 0$ is some universal constant, and we note that the inequality \eqref{diffest} holds for all $(\log n)^{-\frac{1}{3}} \le \epsilon < 1 $. This completes the proof.
\end{proof}

\begin{rem}
    Take $a \in (0,4).$ By replacing the length of each subinterval from $b_{n} = n (\log n)^{-2}$ to $n (\log n)^{-a}$ and carrying out exactly the same proof as for \eqref{slowvar}, we have 
    \begin{equation*}
    \phi \big( n (\log n)^{-a} \big) = \phi (n) \Big\{ 1 + O \Big( (\log n)^{-\frac{1}{2}} \log \log n \Big) \Big\}.
\end{equation*}
\end{rem}

In the following lemma, we show that the condition (A1) of Lemma \ref{exactlem} holds.

\begin{lem}\label{3-lem2}
Let $d=4$. $\phi (n)$ is slowly varying. 
\end{lem}

\begin{proof}
Fix a real number $p>1$. It is sufficient to prove the following equality. 
\begin{equation}\label{slowvar-2}    \phi \big( n/p \big) \sim \phi (n).
\end{equation}
However, by replacing the length of each subinterval from $b_{n}$ to $n/p$ and arguing as in the proof of \eqref{expec}, we have 
$$\big| n \phi (n) - n \phi (n/p) \big| \le C_{p} \Big\{  n (\log n)^{-6} + n (\log n)^{-1} \log \log n \Big\},  $$
for some constant $C_{p} >0$ depending on $p$.
 Dividing both sides by $n \phi (n)$ and using the fact that $\phi(n) \ge c (\log n)^{-\frac{1}{2}}$ by \eqref{cutbound-1}, we obtain \eqref{slowvar-2}.
\end{proof}

Therefore, to prove Theorem \ref{LE-result}, it suffices to verify the condition \eqref{A2} in (A2). 

\vspace{3mm}

To this end, we define events corresponding to the events $A^{1}, \cdots, A^{11}$ introduced in Section 4.2 of \cite{exact}. Take $\epsilon \in (0,1)$. For $100 \le l \le 2M$ and $20 \le j \le l-20$, we define the events $W(1), \cdots, W(9)$ as follows.

\begin{align*}
&W (1) =\big\{ I^{1} (j) \leftrightarrow I^{2} (l-j) \big\}, \\
&W (2) = \{ L^{\ast}= l \}, \\
&W (3) = \Big\{  \text{there exists just  one pair } (j,k) \text{ which attains the maximum of } L^{\ast}  \Big\}, \\
&W (4) = \Big\{ \exists T \in I^{1}(j+1) \text{ such that } T \text{ is a cut time up to } 2 n, \\
&  \ \ \ \ \ \ \  \ \ \ \ \ \ \ \   S[0, t^{1}_{j}] \cap S [n, 2n ] = \emptyset, \ \ \ S\big[ t^{1}_{j-1}, t^{1}_{j-3} \big] \cap S \big[ t^{2}_{l-j-3}, t^{2}_{l-j} \big] = \emptyset \Big\}, \\
&W (5)  = \Big\{ \exists T' \in I^{2} (l-j+1) \text{ such that } T' \text{ is a cut time up to } 2 n, \\ 
& \ \ \ \ \ \ \ \ \ \ \ \ \ \ \ \ \ S [0,n] \cap S [t^{2}_{l-j}, 2n] = \emptyset, \ \ S\big[t^{1}_{j}, t^{1}_{j-3} \big] \cap S \big[ t^{2}_{l-j-3}, t^{2}_{l-j-1} \big] = \emptyset \Big\}, \\
&W (6) = \Big\{ \exists U \in I^{1} (j-1) \text{ such that } U \text{ is a cut time up to } n \Big\}, \\
&W (7) = \Big\{ \exists U' \in I^{2} (l-j-1) \text{ such that } S[n, U'] \cap S[U'+1, 2n] = \emptyset \Big\}, \\
&W (8) = \Big\{ \text{len} \, \Big(     \text{LE}^{\ast} \big( S [ t^{1}_{j-1}, n ] \big) \Big)  \ge (1 -\epsilon ) j b_{n} \phi (n) \Big\}, \\ 
&W (9) = \Big\{ \text{len} \, \Big(     \text{LE}^{\ast} \big( S [n, t^{2}_{l-j-1}] \big) \Big) \ge (1 -\epsilon ) (l-j) b_{n} \phi (n) \Big\}.
 \end{align*}

 \begin{rem}
As they are not needed here, we do not consider the events corresponding to $A^{6}$ and $A^{7}$ defined in Section 4.2 of \cite{exact}. We also note that $W(4)$ and $W(5)$ include additional conditions relative to $A^{4}$ and $A^{5}$ in \cite{exact}, respectively. These additional conditions will be needed in the proof of Lemma \ref{3-lem3} below.
\end{rem}

The following lemma shows that, on the event $W(1) \cap W(2) \cap \cdots \cap W(9)$, a desired lower bound for $L_{n} + L_{n}' - L_{2n}$ can be obtained. Here, we recall that $L_{n}$ and $L_{n}'$ are defined in \eqref{Ln} and in (D3) of Definition \ref{3-def1}, respectively.

\begin{lem}\label{3-lem3}
Take $\epsilon \in (0,1)$, $100 \le l \le 2M$ and $20 \le j \le l-20$. On $W(1) \cap W (2) \cap \cdots \cap W (9),$ we have 
$$ L_{n} + L_{n}' - L_{2n} \ge (1 - \epsilon ) l b_{n} \phi (n) - 6 b_{n}.$$
\end{lem}

\begin{proof}
We first observe that $b_{n}$ is larger than $20$, since $M = (\log n)^2 \ge 50$.

Suppose that $W(1) \cap W (2) \cap \cdots \cap W (9)$ occurs. We take cut times $T, T', U,$ and $U'$, whose existence is guaranteed in $W(4), W(5), W(6)$ and $W(7)$, respectively. 
It follows from \eqref{CUT-LE} that 
\begin{align}
&L_{n} + L_{n}' - L_{2n} \notag  \\
& = \text{len} \, \Big( \text{LE}^{\ast} \big( S[U, n ] \big) \Big) + \text{len} \, \Big( \text{LE}^{\ast} \big( S[n, U' ] \big) \Big) \notag \\
& \ \ \  + \text{len} \, \Big( \text{LE}^{\ast} \big( S[T, U ] \big) \Big) + \text{len} \, \Big( \text{LE}^{\ast} \big( S[U', T' ] \big) \Big) - \text{len} \, \Big( \text{LE}^{\ast} \big( S[T, T' ] \big) \Big) \notag \\
& \ge \text{len} \, \Big( \text{LE}^{\ast} \big( S[U, n ] \big) \Big) + \text{len} \, \Big( \text{LE}^{\ast} \big( S[n, U' ] \big) \Big) - \text{len} \, \Big( \text{LE}^{\ast} \big( S[T, T' ] \big) \Big) \notag \\
&= \text{len} \, \Big( \text{LE}^{\ast} \big( S[t^{1}_{j-1}, n ] \big) \Big) + \text{len} \, \Big( \text{LE}^{\ast} \big( S[n, t^{2}_{l-j-1} ] \big) \Big) \notag \\
& \ \ \ - \text{len} \, \Big( \text{LE}^{\ast} \big( S[t^{1}_{j-1}, U ] \big) \Big) - \text{len} \, \Big( \text{LE}^{\ast} \big( S[U', t^{2}_{l-j-1} ] \big) \Big) - \text{len} \, \Big( \text{LE}^{\ast} \big( S[T, T' ] \big) \Big) \notag \\ 
&\ge (1 -\epsilon ) l b_{n} \phi (n) - 2 b_{n} - \text{len} \, \Big( \text{LE}^{\ast} \big( S[T, T' ] \big) \Big), \label{key-lower}
\end{align}
where in the last inequality we used the fact that the conditions for $W(8)$ and $W(9)$ are satisfied and that the length of $\text{LE}^{\ast} \big( S[a, b] \big)$ is always bounded above by $b-a$.

We will show below that the length of $ \text{LE}^{\ast} \big( S[T, T'] \big) $ is bounded above by $4 b_{n}$. Given that the graph distance or the effective resistance between $S(T)$ and $S(T')$ on $S[T, T']$ is denoted by $X$, due to the occurrence of the event $W(1)$, the triangle inequality ensures that $X \le 4 b_{n}$. However, since $L_{n}$, unlike the graph distance or the effective resistance on $S[0,n]$, does not satisfy the properties of a metric, this is not obvious. In fact, this inequality for the length of  $\text{LE}^{\ast} \big( S[T, T'] \big)$ relies essentially on a property specific to the LLF loop-erasure, which does \emph{not} necessarily hold for the length of the chronological loop-erasure of $S[T, T']$.

Since $W(1)$ occurs, 
$$ A = \big\{ (u, u') \in I^{1} (j) \times I^{2} (l-j) \ : \ S (u) = S( u' ) \big\} \neq \emptyset. $$
We choose $(u_{1}, u_{2} ) \in A$ as follows:
\begin{itemize}
\item let 
$$ r_{1} = \max \big\{ u'- u \ : \ (u,u') \in A \big\}, $$
\item set 
$$ A(1) = \big\{ (u, u') \in A \, : \, u'-u = r_{1} \big\}, $$
\item let 
\begin{align}
&u_{1} = \min \big\{ u \in I^{1} (j) \, : \, \exists u' \in I^{2} (l-j) \, \text{ such that } \, (u, u') \in A(1) \big\}, \, \text{ and } \notag \\
& \ \ \ \ \ \ \ \ \ \ \ \ \ \ \ \ \ \ \ \ \ \  \text{choose } \, u_{2} \, \text{ so that } \, (u_{1}, u_{2} ) \in A(1). \label{const0}
\end{align}
\end{itemize}
That is, $(u_{1}, u_{2}) \in A$ is a pair in $A$ that maximizes the difference among all pairs in $A$, and among such pairs it is the leftmost. 

We aim to check that, in the LLF loop-erasure of $S[T, T']$, the loop $S[u_{1}, u_{2}]$ is the first loop to be erased.
To this end, by the second conditions of the events $W(4)$ and $W(5)$, 
\begin{align}
&\text{no pair $(u,u')$ satisfies either of the following two conditions:} \notag \\
& \ \  \ \ \ \ \text{$u \in [T, t^{1}_{j} ]$ and $u' \in [n, T']$ such that $S(u) = S(u')$,} \label{const1} \\
& \ \ \ \ \ \ \text{$u \in [T, n ]$ and $u' \in [t^{2}_{l-j}, T']$ such that $S(u) = S(u')$.}  \label{const2}
\end{align}
Recall that $T$ and $T'$ satisfy $T \in I^{1}(j+1)$ and $T' \in I^{2}(l-j+1)$, and that $l$ and $j$ satisfy $100 \le l \le 2M$ and $20 \le j \le l-20$. In this situation, when we perform the LLF loop-erasure for $S[T, T']$, let the first loop removed be $S[u, u']$, where $T \le u < u' \le T'$ are times such that $S(u)=S(u')$.
Depending on the relative positions of $(u,u')$ and $n$, the following three cases are possible. 

\vspace{3mm}

\underline{{\bf Case 1:}} $u < u' \le n. $

\vspace{2mm}

\hspace{-7mm} In this case, $u-u' \le (j+1)b_{n}$, while $u_{2}-u_{1} \ge (j-1)b_{n} + (l-j-1)b_{n} \ge (j+18)b_{n}$. This contradicts the definition of the LLF loop-erasure, which removes the largest loops first.

\vspace{3mm}

\underline{{\bf Case 2:}} $n \le u < u'. $

\vspace{2mm}

\hspace{-7mm} In this case, $u-u' \le (l-j+1)b_{n}$, while $u_{2}-u_{1} \ge (j-1)b_{n} + (l-j-1)b_{n} \ge (l-j+18)b_{n}$. This contradicts the definition of the LLF loop-erasure, which removes the largest loops first.

\vspace{3mm}

\underline{{\bf Case 3:}} $u < n < u'. $

\vspace{2mm}

\hspace{-7mm} Since there is no pair of times satisfying either \eqref{const1} or \eqref{const2}, it follows that $u > t^{1}_{j}$ and $u' < t^{2}_{l-j}$. Suppose that $u \in [ t^{1}_{j-1}, t^{1}_{j-3} ].$ By the third condition of $W(4)$ and the constraint \eqref{const2}, we have $u' < t^{2}_{l-j-3}.$ Thus, $u'- u \le (j-1) b_{n} + (l-j-3) b_{n} = (l-4) b_{n}$, while $u_{2}-u_{1} \ge (j-1)b_{n} + (l-j-1)b_{n} = (l-2) b_{n}.$ This contradicts the definition of the LLF loop-erasure, which removes the largest loops first. Similarly, using the third condition of $W(5)$ and the constraint \eqref{const1}, if $u' \in [ t^{2}_{l-j-3}, t^{2}_{l-j-1} ],$ we have $u' - u \le (j-3) b_{n} + (l-j-1) b_{n} = (l-4) b_{n}.$ This also leads to a contradiction. 

Next,  suppose that $u \in [t^{1}_{j-3}, n]$. Since $u' < t^{2}_{l-j}$, we have $u' - u \le (j-3) b_{n} + (l-j) b_{n} = (l-3) b_{n} < (l-2) b_{n} \le u_{2}- u_{1}$. This is a contradiction. Similarly, if $u' \in [n, t^{2}_{l-j-3}]$, since $u > t^{1}_{j}$,  we have $u'-u \le j b_{n} + (l-j-3) b_{n} = (l-3) b_{n} < (l-2 ) b_{n} \le u_{2} -u_{1}$, which leads to a contradiction again.

Hence, we must have $u \in I^{1}(j)$ and $u' \in I^{2}(l-j)$. Since $S(u) = S(u'),$ this implies $(u, u') \in A$. However, as in \eqref{const0}, $(u_{1},u_{2})$ is chosen as the pair in $A$ with the largest difference and the leftmost position. Therefore, by the definition of the LLF loop-erasure given in Definition \ref{LLFLE}, we must have $(u,u')=(u_{1},u_{2})$.

\vspace{3mm}

As a result, it has been verified that in the LLF loop-erasure of $S[T, T']$, the loop $S[u_{1}, u_{2}]$ is the first loop to be erased. That is, in Definition \ref{LLFLE}, when we take $\lambda = S[T, T']$, $\lambda_{1}$ in (ii) coincides with $S[T, u_{1}] \oplus S[u_{2}, T']$.
 Thus, we have
$$ \text{LE}^{\ast} \big( S[T, T'] \big) = \text{LE}^{\ast} \big( S[T, u_{1}] \oplus S[u_{2}, T'] \big). $$
Returning to \eqref{key-lower}, since $|T- u_{1}| + |T' - u_{2} | \le 4 b_{n},$ it follows that
$$ \text{len} \, \Big( \text{LE}^{\ast} \big( S[T, T' ] \big) \Big) \le 4 b_{n}, $$ as desired.
\end{proof}

\begin{rem}
We use the same notation as in the proof of Lemma \ref{3-lem3}. The length of the chronological loop-erasure of $S[T,T']$ is not necessarily small. This is because $S[u_{1},u_{2}]$ is not necessarily removed by the chronological loop-erasing procedure.   
\end{rem}

With Lemma \ref{3-lem3} in mind, we have
 \begin{align}
 &E (L_n + L_{n}' - L_{2n} ) \notag \\
 &= \sum_{l=2}^{2M} E \Big( L_n + L_{n}' - L_{2n} \ ; \  W(2) \Big) \notag \\
 &\ge \sum_{l=2}^{2M} E \Big( L_n + L_{n}' - L_{2n} \ ; \  W(2) \cap W (3) \Big) \notag \\
 &= \sum_{l=2}^{M} \sum_{j=1}^{l-1} E \Big( L_n + L_{n}' - L_{2n} \ ; \  W(1) \cap W(2) \cap W (3) \Big) \notag \\
 & \ \ \ \ \ + \sum_{l=M+1}^{2M}  \sum_{j=l-M}^{M} E \Big( L_n + L_{n}' - L_{2n} \ ; \  W(1) \cap W(2) \cap W (3) \Big) \notag \\
 &\ge \sum_{l=\epsilon M}^{M} \sum_{j=\epsilon l}^{(1-\epsilon)l} E \Big( L_n + L_{n}' - L_{2n} \ ; \  W(1) \cap W(2) \cap W (3) \Big) \notag \\
 & \ \ \ \ \ + \sum_{l=(1+\epsilon)M}^{2M}  \sum_{j=l-M}^{M} E \Big( L_n + L_{n}' - L_{2n} \ ; \  W(1) \cap W(2) \cap W (3) \Big) \notag \\
 &\ge \sum_{l=\epsilon M}^{M} \sum_{j=\epsilon l}^{(1-\epsilon)l} E \Big( L_n + L_{n}' - L_{2n} \ ; \  W(1) \cap W(2) \cap \cdots \cap  W (9) \Big) \notag \\
 & \ \ \ \ \ + \sum_{l=(1+\epsilon)M}^{2M}  \sum_{j=l-M}^{M} E \Big( L_n + L_{n}' - L_{2n} \ ; \  W(1) \cap W(2) \cap \cdots \cap W(9) \Big) \notag \\
 &\ge \sum_{l=\epsilon M}^{M} \sum_{j= \epsilon  l}^{(1 - \epsilon) l} \Big\{ (1- \epsilon) l b_{n} \phi (n) - 6 b_{n} \Big\} \notag  \, P \Big( W (1) \cap W (2) \cap \cdots \cap W (9) \Big) \notag  \\
 & \ \ \ \ \ \ + \sum_{l= (1+ \epsilon ) M}^{2M} \sum_{j= l-M}^{M} \Big\{ (1- \epsilon) l b_{n} \phi (n) - 6 b_{n} \Big\} \, P \Big( W (1) \cap W (2) \cap \cdots \cap W (9) \Big), \label{3-lower}
 \end{align}
 where we observe that if $\epsilon \in (0,1)$ is fixed and $n$ is taken sufficiently large, then $l$ and $j$ considered in the sum on the right-hand side of \eqref{3-lower} satisfy $100 \le l \le 2M$ and $20 \le j \le l-20$. (Recall that $M= (\log n)^{2}$, as defined in Definition \ref{3-def1}.) 

 To give a lower bound on the right-hand side of \eqref{3-lower}, we need the following lemma.

 \begin{lem}\label{3-lem4}
 There exists $r > 0$ such that the following holds. Take $\epsilon \in (0,1)$. There exist $C = C_{\epsilon} > 0$ and $N = N_{\epsilon} \ge 1$ depending on $\epsilon$ such that for $n \ge N$ and for $l$ and $j$ satisfying either
 $$ \epsilon M  \le l \le M \ \ \text{ and } \ \  \epsilon l \le j \le (1- \epsilon ) l $$
 or 
 $$ (1 + \epsilon ) M \le l \le 2M \ \ \text{ and } \ \  l-M \le j \le M,$$
 we have 
 $$ P \Big( W (1) \cap W (2) \cap \cdots \cap W (9) \Big) \ge P \big( W (1) \big) \Big\{ 1 - C (\log n)^{-\frac{1}{2}} (\log \log n)^{r} \Big\}.$$
 \end{lem}
 
 \begin{proof}
 It suffices to show that for each $2 \le i \le 9$, 
 \begin{equation}\label{3-lem4-1}
 P \Big( W(1) \cap W(i)^{c} \big) \le C P \big( W(1) \big)  (\log n)^{-\frac{1}{2}} (\log \log n)^{r} .
 \end{equation}
It follows from Proposition 2.3 in \cite{DS-SW-4d} that the probability of the event $W(1)$ satisfies the following equality: 
\begin{equation}\label{W1prob}
P \big( W(1) \big) = \frac{1}{2 \log n }  \log \Big[ 1 + \frac{1}{l (l-2)} \Big] \Big\{ 1 + O \Big(   l^{-1} (\log n)^{-\frac{3}{25}} \Big) \Big\} \ \ \  \text{ for } \ \ l \ge 3,
\end{equation}
where the implicit constant in $O$ does not depend on $l$, $j$ and $n$. 
In particular, with the assumption for $l$, we have 
\[
    P(W(1))\asymp l^{-2}(\log n)^{-1}\asymp_{\epsilon} (\log n)^{-5}.
\]
Let $k = l-j$. 
 
We begin with $i=2$. Suppose that $W(1) \cap W(2)^{c}$ occurs. This implies $L^{\ast} \ge l+1$. Therefore, either 
$$ S[0, t^{1}_{j}] \cap S[n, 2n] \neq \emptyset \ \ \text{ or } \ \ S[ 0, n] \cap S[ t^{2}_{k}, 2n ] \neq \emptyset$$
occurs. 
To give an upper bound of 
\[
    P \Big( W(1) \cap \big\{ S[0, t^{1}_{j}] \cap S[n, 2n] \neq \emptyset \big\} \Big),
\] 
we apply Propositions 4.1 and 4.3 of \cite{slowly} as follows. (We only handle the probability of the first event since that of the latter event can be obtained using the same argument.) 
Let $S^i~(i=1,2)$ be independent simple random walks on $\mathbb{Z}^4$ and let $P^x_i$ be the law of $S^i$ conditioned $S^i(0)=x$. 
We also denote by $P$ the joint law of $S^1$ and $S^2$ conditioned that both walks start at the origin. 
By Theorem 1.2.1 and Proposition 1.5.10 of \cite{Lawb}, if we define 
\[
    \Lambda_1=\left\{\mathrm{dist}(S^1(b_n),S^2[0, 2n])\ge \sqrt{n}(\log n)^{-30}\right\},
\]
then we have 
\[
    P(\Lambda_1^c)\le C(\log n)^{-10},
\]
for some $C>0$. 
Now we define an $S^{2}[0, 2n]$-measurable random variable $Y$ by
\[
    Y=\max\left\{P_1^x(S^1[0,n]\cap S^2[0, 2n]\neq\emptyset)\mathrel{:}x\in\mathbb{Z}^4,~\mathrm{dist}(x,S^2[0, 2n])\ge \sqrt{n}(\log n)^{-30}\right\},
\]
where $\mathrm{dist}(A,B)$ denotes the distance between sets $A$ and $B$ with respect to the Euclidean metric on $\mathbb{R}^4$. 
Propositions 4.1 and 4.3 of \cite{slowly} guarantees that there exists $r>0$ such that 
\[
    P_2^0\left(Y\ge \frac{(\log\log n)^r}{\log n}\right)\le C(\log n)^{-10}.
\]
Let $\Lambda_2=\{Y\le (\log n)^{-1}(\log\log n)^r\}$ and note that $\Lambda_2$ is measurable with respect to $S^2[0, 2n]$. 
By the Markov property at $t_{j-1}^1$ and the reversibility, we have 
\begin{align}
    P&\big(W(1) \cap \big\{S[0, t^{1}_{j}] \cap S[n, 2n] \neq \emptyset \big\}\big) \notag\\
    &\le P\left(W(1)\cap \big\{S[0, t^{1}_{j}] \cap S[t^{1}_{j-1}, 2n] \neq \emptyset \big\}\cap \big\{\mathrm{dist}(S(t^1_j),S[t^{1}_{j-1},2n])\ge \sqrt{n}(\log n)^{-30}\big\}\right) \notag\\
    &\qquad +P(\Lambda_1^c) \notag\\
    &\,\le C P \big( I^1(j)\leftrightarrow I^2(l-j) \big) (\log n)^{-1}(\log \log n)^r     +P(\Lambda_1^c)+P(\Lambda_2^c)  \notag\\
    &\le CP(W(1)) \, (\log n)^{-1}(\log \log n)^r.\label{ineq-1-cap-2c}
\end{align}
For $i=3$, note that 
\begin{align*}
    W(1)\cap W(3)^c
        =&\left(\{I^1(j)\leftrightarrow I^2(l-j)\}\cap\{\exists j'\neq j\mbox{ such that }I^1(j')\leftrightarrow I^2(l-j')\}\right) \\
        &\quad \cup\left\{\begin{gathered}
        \exists l'>l,\exists j',j''\mbox{ with }j'<j''\mbox{ such that }    \\
        I^1(j')\leftrightarrow I^2(l'-j'),I^1(j'')\leftrightarrow I^2(l'-j'')\end{gathered}\right\},
\end{align*}
and each event in the union on the right-hand side can be handled by the same argument as \eqref{ineq-1-cap-2c}. 

For $i=4$, let \[
    \Lambda_3=\left\{\mbox{There is no cut time of }S[t_{j+1}^1,t_j^1]\mbox{ in }[t_{j+1}-b_n(\log b_n)^{-6},t_j^1]\right\},
\]
and note that 
\[
    \left\{\mbox{There is no cut time up to }2n\mbox{ in }I^1(j+1)\right\}
        \subset \Lambda_3 \cup \{S[0,t_{j+1}^1]\cap S[t_j^1,2n)\neq\emptyset\}.
\]
Thus we have 
\begin{align}
    P(W(1)\cap W(4)^c)&\le P(W(1)\cap \Lambda_3)+P(W(1)\cap\{S[0,t_{j+1}^1]\cap S[t_j^1,2n]\neq\emptyset\})\notag \\
        &\quad + P(W(1)\cap\{S[0,t_j^1]\cap S[n,2n]\neq\emptyset\})\notag    \\
        &\quad + P(W(1)\cap\{S[t^1_{j-1},t_{j-3}^1]\cap S[t^2_{l-j-3},t^1_{l-j}]\neq\emptyset\}).\label{ineq-1-cap-4c}
\end{align}
Since $W(1)$ and $\Lambda_3$ are independent, the first term on the right-hand side of \eqref{ineq-1-cap-4c} is bounded above by 
\[
    P(W(1))P(\Lambda_3)\le CP(W(1))(\log n)^{-1}\log \log n,
\]
using Lemma 7.7.4 of \cite{Lawb}. 
The remaining terms can be handled the same way as we did for \eqref{ineq-1-cap-2c}. Note that this strategy also works for the case $i=5$. 

For $i=6,7,8,9$, we first deal with $W(6)$. 
Note that defining 
\begin{gather*}
    W^*=\{S^1[0,b_n]\cap S^2[(l-2)b_n,(l-1)b_n]\neq\emptyset\}, \\
    W^{**}=\left\{\begin{gathered}\forall U\in[b_n-b_n(\log b_n)^{-6},b_n],
    \\ (S^1[0,n-(j-1)b_n]\cap S^2[0,U])\cap(S^2[U+1,(j-1)b_n])\neq \emptyset\end{gathered}\right\},
\end{gather*}
we have 
\[
    P(W(1)\cap W(6)^c)=P(W^*\cap W^{**}). 
\]
Our strategy for this case is again to apply Propositions 4.1 and 4.3 of \cite{Lawb}. 
However, in the straightforward use of these propositions, that is, by ``freezing" $S^1[0,b_n]$, the product of $(\log n)^{-1}(\log \log n)^r$ and $P(W^{**})$ does not guarantee our desided bound (recall \eqref{W1prob}). 
With this in mind, let us also define 
\begin{gather*}
    \Lambda_4=\left\{\mathrm{diam}(S^1[0,b_n])\vee \mathrm{diam}(S^2[(l-2)b_n,(l-1)b_n])\le \sqrt{b_n}(\log \log n)^2\right\},\\
    \Lambda_5=\{|S^2((l-2)b_n)|\le 2\sqrt{b_n}(\log\log n)^2\},
\end{gather*}
so that we have 
\[
    W^*\cap \Lambda_4\subset \Lambda_5;\quad P(\Lambda_4^c)\le (\log n)^{-10},
\]
by Lemma 1.5.1 of \cite{Lawb}. 
Now we will apply the argument of \cite{slowly} with a slight modification. 
Let 
\begin{gather*}
    \Lambda^*_1=\left\{\mathrm{dist}(S^2((l-2)b_n),S^1[0,b_n])\ge \sqrt{n}(\log n)^{-30}\right\},\\
    \Lambda^*_2=\{Y^*\le (\log n)^{-1}(\log \log n)^r\},
\end{gather*}
where we set 
\[
    Y^*=\max\left\{P_2^x(S^1[0,b_n]\cap S^2[0,b_n]\neq\emptyset)\mathrel{:}~
    \begin{gathered}x\in\mathbb{Z}^4,\\
    \mathrm{dist}(x,S^1[0,b_n])\ge \sqrt{n}(\log n)^{-30}
    \end{gathered}\right\}.
\]
By the same argument as the case $i=2$, we have 
\[
    P(\Lambda_1^*)\le C(\log n)^{-10};\quad P^0_1\left(Y^*\ge \frac{(\log \log n)^r}{\log n}\right)\le C(\log n)^{-10},
\]
Thus we have 
\begin{align*}
    P(W(1)\cap W(6)^c)&\le P(W^*\cap W^{**}\cap \Lambda_1^*\cap \Lambda_2^*\cap \Lambda_4)+C(\log n)^{-10} \\
        &\le P(W^*\cap W^{**}\cap \Lambda_1^*\cap \Lambda_2^*\cap \Lambda_5)+C(\log n)^{-10}  \\
        &\le \frac{(\log \log n)^r}{\log n}\times P(W^{**}\cap \Lambda_5)+C(\log n)^{-10}  \\
        &\le \frac{(\log \log n)^r}{\log n}\max_{x\in \mathbb{Z}^4}P^x(|S((l-j-1)b_n)|\le 2\sqrt{b_n}(\log\log n)^2)    \\
        &\quad \times P(W^{**})+C(\log n)^{-10}, 
\end{align*}
where the third and the last inequality follows from the application of the Markov property to $S^2$ at time $(l-2)b_n$ and $(j-1)b_n$, respectively. 
Moreover, from Theorem 1.2.1 of \cite{Lawb} and the assumption on $l$ and $j$, it follows that 
\begin{align*}
    \max_{x\in \mathbb{Z}^4}P^x(|S((l-j-1)b_n)|&\le 2\sqrt{b_n}(\log\log n)^2) \le C\frac{b_n^2(\log \log n)^8}{(l-j-1)^2b_n^2} \\
    &\le C(\log n)^{-4}(\log \log n)^8.
\end{align*}
Combining this with Lemma 7.7.4 of \cite{Lawb}, we obtain 
\[
    P(W(1)\cap W(6)^c)\le \frac{C(\log\log n)^{r+9}}{(\log n)^6}\le CP(W(1))(\log n)^{-1}(\log\log r)^{r+9}.
\]
The same argument covers the case $i=7$. For $i=8$, let 
\[
    W^{***}=\left\{\mathrm{len}(\mathrm{LE}^*(S^2[0,(j-1)b_n]))\ge (1 -\epsilon) j b_{n}\phi(n)\right\}.
\]
Recalling the argument above, we have 
\begin{align*}
    P(W(1)\cap W(8)^c)
        &\le P(W^*\cap W^{***}\cap \Lambda_1^*\cap \Lambda_2^*\cap \Lambda_5)+C(\log n)^{-10}  \\
        &\le \frac{(\log \log n)^r}{\log n}\times P(W^{***}\cap \Lambda_5)+C(\log n)^{-10}  \\
        &\le \frac{(\log \log n)^{r+8}}{(\log n)^5}P(W^{***})+C(\log n)^{-10}.
\end{align*}
Now \eqref{3-lem1-1} gives 
\[
    P(W^{***})\le C(\log n)^{-\frac{1}{2}}\log\log n,
\]
and together with the above inequality, we obtain \eqref{3-lem4-1} for $i=8$, so do we for $i=9$. This finishes the proof. 
\end{proof}

We are now ready to prove the condition (A2) of Lemma \ref{exactlem}. The following lemma establishes this.

 \begin{lem}\label{last}
 There exists a universal constant $C>0$ such that the following holds. Let $\epsilon \in (0,1)$. Then there exists $N=N_{\epsilon} \ge 1$ depending on $\epsilon$ such that for all $n \ge N$,
 $$ E (L_{n} + L_{n}' - L_{2n} ) \ge (1 -  C  \epsilon ) n (\log n)^{-1} \phi (n) \log 2.$$

Consequently, by dividing both sides by $2n$, the condition \eqref{A2} holds.
 
\end{lem}

\begin{proof}

Let $C_{\epsilon}$ and $N_{\epsilon}$ be as in Lemma \ref{3-lem4}. By replacing $N_{\epsilon}$, if necessary, with a larger value depending on $\epsilon$, we may assume that 
\begin{equation}\label{largen}
    \epsilon M \ge 100,  \ \ \  \epsilon^{2} M \ge 20 \ \ \ \text{and} \ \ \  P \big( W(1) \big) \ge \frac{1 - \epsilon}{2 l^{2} \log n}
\end{equation}
whenever $n \ge N_{\epsilon}$ and $l \ge \epsilon M$. 
Combining Lemma \ref{3-lem3}, \eqref{3-lower} and Lemma \ref{3-lem4}, if $n \ge N_{\epsilon}$, we have 
\begin{align*}
&E (L_n + L_{n}' - L_{2n} ) \notag \\
 &\ge \sum_{l=\epsilon M}^{M} \sum_{j= \epsilon  l}^{(1 - \epsilon) l} \Big\{ (1- \epsilon) l b_{n} \phi (n) - 6 b_{n} \Big\}   \, P \Big( W (1) \cap W (2) \cap \cdots \cap W (9) \Big) \notag  \\
 &+ \sum_{l= (1+ \epsilon ) M}^{2M} \sum_{j= l-M}^{M} \Big\{ (1- \epsilon) l b_{n} \phi (n) - 6 b_{n} \Big\} \, P \Big( W (1) \cap W (2) \cap \cdots \cap W (9) \Big) \\
 &\ge \sum_{l=\epsilon M}^{M} \sum_{j= \epsilon  l}^{(1 - \epsilon) l} \Big\{ (1- \epsilon) l b_{n} \phi (n) - 6 b_{n} \Big\}   \, P \big( W (1) \big) \Big\{ 1 - C_{\epsilon} (\log n)^{-\frac{1}{2}} (\log \log n)^{r} \Big\} \\
 &+ \sum_{l= (1+ \epsilon ) M}^{2M} \sum_{j= l-M}^{M} \Big\{ (1- \epsilon) l b_{n} \phi (n) - 6 b_{n} \Big\} \, P \big( W (1) \big) \Big\{ 1 - C_{\epsilon} (\log n)^{-\frac{1}{2}} (\log \log n)^{r} \Big\},
\end{align*}
for some universal constant $r >0$. Here, since the first two conditions in \eqref{largen} are satisfied, observe that the indices $l$ and $j$ in the above sums satisfy $100 \le l \le 2M$ and $20 \le j \le l - 20$.

Substituting the lower bound for the probability of $W(1)$ given in \eqref{largen} into this, the right-hand side above is bounded below by 
\begin{align*}
& \sum_{l=\epsilon M}^{M} \sum_{j= \epsilon  l}^{(1 - \epsilon) l} \Big\{ (1- \epsilon) l b_{n} \phi (n) - 6 b_{n} \Big\}  \frac{1- \epsilon}{2 l^{2}}  (\log n)^{-1} \Big\{ 1 - C_{\epsilon} (\log n)^{-\frac{1}{2}} (\log \log n)^{r} \Big\} \\
 &+ \sum_{l= (1+ \epsilon ) M}^{2M} \sum_{j= l-M}^{M} \Big\{ (1- \epsilon) l b_{n} \phi (n) - 6 b_{n} \Big\} \, \frac{1- \epsilon}{2 l^{2}}  (\log n)^{-1} \Big\{ 1 - C_{\epsilon} (\log n)^{-\frac{1}{2}} (\log \log n)^{r} \Big\}.
\end{align*}
By carefully evaluating these sums, we see that 
\begin{align*}
&\text{(S1) \ \  \ \ }\sum_{l=\epsilon M}^{M} \sum_{j= \epsilon  l}^{(1 - \epsilon) l}      \frac{b_{n}}{2 l^{2}}  (\log n)^{-1} \le C n (\log n)^{-3} \log \epsilon^{-1}, \\
&\text{(S2) \ \  \ \ }\sum_{l= (1+ \epsilon ) M}^{2M} \sum_{j= l-M}^{M}      \frac{b_{n}}{2 l^{2}}  (\log n)^{-1} \le C n (\log n)^{-3}, \\
&\text{(S3)  \ \ \ \ } \sum_{l=\epsilon M}^{M} \sum_{j= \epsilon  l}^{(1 - \epsilon) l}  (1- \epsilon) l b_{n} \phi (n)   \frac{1- \epsilon}{2 l^{2}}  (\log n)^{-1}  \\ 
& \ \ \ \  \ \ \ \ \ \ \ \ \ \ \ \  = \frac{(1- \epsilon)^{3} (1- 2 \epsilon)}{2} n \phi (n) (\log n)^{-1} + O \Big( n \phi (n) (\log n)^{-3} \log \epsilon^{-1} \Big), \\
&\text{(S4)  \ \ \ \ } \sum_{l= (1+ \epsilon ) M}^{2M} \sum_{j= l-M}^{M}  (1- \epsilon) l b_{n} \phi (n)  \frac{1- \epsilon}{2 l^{2}}  (\log n)^{-1}  \\ 
& \ \ \ \  \ \ \ \ \ \ \ \ \ \ \ \  = \frac{(1 - \epsilon)^{2}}{2} \Big\{ 2 \log  \frac{2}{1 + \epsilon} - (1 - \epsilon) \Big\} n \phi (n) (\log n)^{-1} + O \Big( n \phi (n) (\log n)^{-3}   \Big).
\end{align*}
Note that the term $C_{\epsilon} (\log n)^{-\frac{1}{2}} (\log \log n)^{r}$ can be neglected.
If we add the sum in (S3) to the sum in (S4) and $n$ is taken sufficiently large depending on $\epsilon$, we see that $E (L_{n} + L_{n}' - L_{2n})$ is bounded below by
$$ (1 -  C  \epsilon ) n (\log n)^{-1} \phi (n) \log 2$$
for some universal constant $C >0$, as desired.
\end{proof}

\vspace{3mm}

\begin{proof}[Proof of Theorem \ref{main}]
The lower bound follows from \eqref{cutbound-1}. For the upper bound, Lemma \ref{3-lem2} and Lemma \ref{last} allow us to apply Lemma \ref{exactlem}, which yields the desired upper bound.
\end{proof}

\section{Convergence to Brownian motion}

In this section, we prove that LLF LERW, when appropriately rescaled in space and time, converges to four-dimensional Brownian motion. We begin by introducing the necessary notation.

We use the following partition of the time interval $[0,n]$.
\begin{dfn}\label{convdef}
Let $e_{n} = n (\log n)^{-\frac{1}{8}}$ and $K_{n} = (\log n)^{\frac{1}{8}}$. We partition the time interval $[0,n]$ into $K_{n}$ subintervals of equal length $e_{n}$ as follows:
\begin{itemize}
    \item For $j = 0, 1, \cdots, K_{n}$, let $s_{j} = j e_{n}$. 
    \item For $j = 1, 2, \cdots, K_{n}$, set 
    $$\gamma_{j} = \text{LE}^{\ast} \big( S [s_{j-1}, s_{j} ] \big) \ \ \ \text{ and }  \ \ \ L_{j} = \text{len} \, \big( \gamma_{j} \big).$$
    
\end{itemize}

\end{dfn}

Recall that $X_{n}(t)$ is defined in \eqref{xnt}.
The aim of this section is to prove the following theorem.

\begin{thm}
    Let $d=4$. As $n \to \infty$, 
    $\big( X_{n}(t) \big)_{t \in [0,1]} $ converges to a standard Brownian motion $\big( B(t) \big)_{t \in [0,1]}$ in $\mathbb{R}^{4}$ started at the origin.
\end{thm}

\begin{proof}
Recall that $a_{n} = n (\log n)^{-6}$. Using the notation introduced in Definition \ref{convdef}, for each $j = 1, 2, \cdots, K_{n}$, let $V (j)$ denote the event that there exists a cut time up to $n$ in both intervals $[s_{j} - a_{n}, s_{j}]$ and $[s_{j}, s_{j} + a_{n}]$. By Lemma 7.7.4 of \cite{Lawb}, we have 
$$ P ( V(j) ) \ge 1 - C (\log n)^{-1} \log \log n,$$
    for all $j = 1,2, \cdots , K_{n}.$ Therefore, if we write 
    $$ V = \bigcap_{j=1}^{K_{n}} V (j),$$
it follows that 
$$ P (V) \ge 1 - C (\log n)^{-\frac{7}{8}} \log \log n.$$


Using again the notation introduced in Definition \ref{convdef}, define the event $F(j)$ by
$$F (j) = \Big\{ |L_{j} - e_{n} \phi (n) | \le e_{n} \phi (n) (\log n)^{-\frac{1}{4}} \Big\}.$$
Setting $\epsilon = (\log n)^{-\frac{1}{4}}$ and applying \eqref{diffest}, we see that, for each $j$,
$$ P (F(j)) \ge 1 - C (\log n)^{-\frac{1}{4}} \log \log n,$$
holds. Here, $C > 0$ is a universal constant.
Therefore, if we write 
    $$ F = \bigcap_{j=1}^{K_{n}} F (j),$$
it follows that 
$$ P (F) \ge 1 - C (\log n)^{-\frac{1}{8}} \log \log n.$$

Define the event $W$ by
$$W = \Big\{ \max \big\{ | S (j) - S (k) | \ : \ 0 \le j \le k \le n, \ k-j \le e_{n} \big\} \le \sqrt{n} (\log n)^{- \frac{1}{32}} \Big\}.$$
By Lemma 1.5.1 of \cite{Lawb}, the probability that $W$ occurs satisfies
$$ P (W) \ge 1 - C e^{- c (\log n)^{-\frac{1}{32}}}. $$

By Lemma 3.1 of \cite{Lawcut}, a standard Brownian motion $B$ in $\mathbb{R}^{4}$ and a simple random walk $S$ on $\mathbb{Z}^{4}$, both starting from the origin, can be defined on the same probability space (that is, they can be coupled) so that
$$P \bigg( \max_{0 \le t \le 1} \bigg| B (t) - \frac{2 S \big( t n \big)}{ \sqrt{n}} \bigg| \ge n^{-\frac{1}{5}} \bigg) \le C e^{-n^{c}},$$
holds for universal constants $c, C > 0$.

We assume that $B$ and $S$ are coupled as above. Suppose that the event $V \cap F \cap W$ occurs. Then
    $$\max_{0 \le t \le 1} \bigg| X_{n} (t) - \frac{2 S ( t n)}{\sqrt{n}} \bigg| \le C  (\log n)^{-\frac{1}{32}},$$
    holds for some universal constant $C > 0$. Consequently, 
    $$\max_{0 \le t \le 1} \big| X_{n} (t) - B(t) \big| \le C  (\log n)^{-\frac{1}{32}},$$
    which yields the desired conclusion.
\end{proof}

\bibliographystyle{plain} 
\bibliography{4d}


\end{document}